\documentclass{amsart}
\usepackage{enumerate, amsfonts,
amssymb,xypic,amsmath,amsthm,pstricks,pst-node,pst-tree,float,epsfig}
\usepackage{caption}
\usepackage{subcaption}

\newlength{\tabwidth}
\newlength{\tabheight}
\newlength{\tabrule}
\newlength{\tabwidthx}
\newlength{\tabheightx}

\def\gentabbox#1#2#3#4{\vbox to \tabheight{\setlength{\tabrule}{#3}%
  \setlength{\tabwidthx}{#1\tabwidth}\addtolength{\tabwidthx}{\tabrule}%

\setlength{\tabheightx}{#2\tabheight}\addtolength{\tabheightx}{-\tabheight}%
  \hbox to #1\tabwidth{%
 \hspace{-0.5\tabrule}\rule{\tabrule}{#2\tabheight}\hspace{-\tabrule}%
    \vbox to #2\tabheight{\hsize=\tabwidthx%
      \vspace{-0.5\tabrule}\hrule width\tabwidthx height\tabrule%
      \vspace{-0.5\tabrule}\vfil%
      \hbox to \tabwidthx{\hss#4\hss}%
        \vfil\vspace{-0.5\tabrule}%
      \hrule width\tabwidthx height\tabrule\vspace{-0.5\tabrule}}%
 \hspace{-\tabrule}\rule{\tabrule}{#2\tabheight}\hspace{-0.5\tabrule}}%
  \vspace{-\tabheightx}}}
\def\genblankbox#1#2{\vbox to \tabheight{\vfil\hbox to
#1\tabwidth{\hfil}}}
\def\tabbox#1#2#3{\gentabbox{#1}{#2}{0.4pt}{\strut #3}}

\catcode`\:=13 \catcode`\.=13 \catcode`\;=13
\catcode`\>=13 \catcode`\^=13
\def:#1\\{\hbox{$#1$}}
\def.#1{\tabbox{1}{1}{$#1$}}
\def>#1{\tabbox{2}{1}{$#1$}}
\def^#1{\tabbox{1}{2}{$#1$}}
\def;{\genblankbox{1}{1}\relax}
\catcode`\:=12 \catcode`\.=12 \catcode`\;=12
\catcode`\>=12 \catcode`\^=7

\newenvironment{tableau}{\bgroup\catcode`\:=13 \catcode`\.=13
  \catcode`\;=13 \catcode`\>=13 \catcode`\^=13
  \setlength{\tabheight}{3ex}\setlength{\tabwidth}{3ex}%
  \def\b##1##2##3{\gentabbox{##1}{##2}{1.2pt}{\vbox{##3}}}%
  \def\n##1##2##3{\gentabbox{##1}{##2}{0.4pt}{\vbox{##3}}}%
  \vbox\bgroup\offinterlineskip}{\egroup\egroup}

\newtheorem{theorem}{Theorem}[section]

\newtheorem{lemma}[theorem]{Lemma}
\newtheorem{proposition}[theorem]{Proposition}

\newtheorem*{theorem*}{Theorem}
\newtheorem*{corollary*}{Corollary}

\theoremstyle{definition}
\newtheorem{definition}[theorem]{Definition}

\theoremstyle{remark}

\newtheorem{example}[theorem]{Example}

\begin{document}
\title{Sign under the domino Robinson-Schensted maps}
\author{Thomas Pietraho}
\date{}
\thanks{We would like to thank Skidmore College for its hospitality during the writing of this manuscript.}
\email{tpietrah@bowdoin.edu} \subjclass[2000]{05E10}
\keywords{Robinson-Schensted, domino tableaux}
\address{Department of Mathematics\\Bowdoin College\\Brunswick,
Maine 04011}
 \maketitle

\begin{abstract} We generalize a formula obtained independently by A.~Reifegerste
and J.~Sj\"ostrand for the sign of a permutation under the classical Robinson-Schensted map to a family of domino Robinson-Schensted algorithms.
\end{abstract}
\section{Introduction}

In their work verifying R.~Stanley's sign imbalance formula, A.~Reifegerste
and J.~Sj\"ostrand independently obtained the following remarkable formula for reading the sign of a permutation from its image under the classical Robinson-Schensted map.  It is based on two tableaux statistics $e$ and $sign$:
\begin{theorem*}[\cite{reifegerste},\cite{sjostrand}]
Consider $w \in S_n$ and let $RS(w) = (P,Q)$ be its image under the classical Robinson-Schensted map.  Then
$$ sign(w) = (-1)^{e} \cdot sign(P) \cdot sign(Q).$$
\end{theorem*}

In \cite{garfinkle1}, D.~Garfinkle, building on the work of D.~Barbasch and D.~Vogan in \cite{barbasch:vogan}, introduced a generalization of the Robinson-Schensted algorithm relating elements of the other classical Weyl groups and same-shape pairs of domino tableaux.  This algorithm was further extended to a one-parameter family of maps $G_r$ by M.~A.~A.~van Leeuwen using domino tableaux with non-empty core.  For large values of $r$, van Leeuwen's maps recover yet another generalization of the Robinson-Schensted maps in this setting introduced by R.~Stanley \cite[\S 6]{stanley:some}.  The aim of this paper is to address the natural question of whether it is again possible to recover the parity of the length function, which we call its sign, from the tableaux images of these maps.

In fact, this has already been done for the Stanley map in \cite{mps} in the more general context of complex reflection groups.  For the family of maps $G_r$, our main result relies on three domino tableaux statistics $d$, $spin$, and $sign$ defined in Section \ref{subsection:stats}.  Let $H_n$ be the Weyl group of type $B_n$.

\begin{theorem*}
Consider $w \in H_n,$ and let $G_r(w) = (P,Q)$ be its image among same-shape standard domino tableaux of rank $r$.  Then
$$ sign(w) = (-1)^{d} \cdot (-1)^{spin(P)+spin(Q)} \cdot sign(P) \cdot sign(Q)$$
\end{theorem*}

The proof involves three steps.  We first verify the equation for large $r$ by translating between Stanley's map and $G_r$ and appealing to the formula established in \cite{mps}.  It is then possible to extend the result to involutions in $H_n$ for arbitrary $r$ using a relation between consecutive maps $G_r$ described in \cite{pietraho:relation}, and finally to to all signed permutations by tracking the behavior of the established sign formula under M.~Ta\c{s}k{\i}n's plactic relations introduced in \cite{taskin:plactic}.

The domino tableaux Robinson-Schensted algorithms appear in the work classifying Kazhdan-Lusztig cells in unequal parameter Iwahori-Hecke algebras of type $B$, see \cite{bgil}.  At least conjecturally, for certain values of the parameter, one-sided cells correspond to plactic and coplactic classes for the maps $G_r$.  As corollary to the above sign formula, we note that the M\"obius function for the Bruhat order, ubiquitous in Kazhdan-Lusztig theory,  is well-behaved with respect to these cells.  First described by  by D.-N.~Verma in \cite{verma:mobius}, the M\"obius function $\mu$ takes the form $$\mu(v,w) = (-1)^{\ell(v)+\ell(w)},$$ where $\ell$ is the length function on the Weyl group.  In type $B$, the values of $\mu$ can be readily read off from the tableaux of $v$ and $w$ arising from the maps $G_r$.  Further,
\begin{corollary*} Consider $x,x',y,y' \in H_n$ and fix a map $G_r$.  Suppose that the left tableaux of the pair $x$ and $y$ as well as $x'$ and $y'$ are the same, and  the right tableaux of the pairs $x,x'$ and $y,y'$ similarly agree. Then
$$\mu(x,y) = \mu(x',y').$$
\end{corollary*}

\section{Preliminaries}

We define the notions of standard and domino tableaux, describe a family of Robinson-Schensted maps, and detail several tableaux statistics which will be necessary for our work.

\subsection{Partitions and tableaux}

Our first objective is to define the notions of standard Young, bi-, and domino tableaux.
A non-increasing sequence of positive integers $\lambda = (\lambda_1, \lambda_2, \ldots, \lambda_t)$ is called a {\it partition} of the integer $n = \sum_i \lambda_i$.  We will write $\lambda \vdash n$ and $|\lambda|=n$.  Partition notation can often be abbreviated by using exponents to denote multiplicity; for instance, $(4,4,3,3,3,1)$ can be written as $(4^2, 3^3,1)$.  We will identify a partition with its Young diagram $[\lambda]$, or a left-justified array of squares containing $\lambda_i$ squares in row $i$.

\begin{figure}[h!]
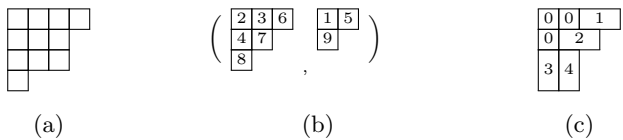

        \begin{subfigure}{0.2\textwidth}
                \centering
\begin{tiny}
$$
\begin{tableau}
:.{}.{}.{}.{}\\
:.{}.{}.{}\\
:.{}.{}.{}\\
.{}
\end{tableau}
$$
\end{tiny}
\vspace{-.2in}
\caption*{(a)}
\end{subfigure}
\hspace{.2in}
\begin{subfigure}{0.2\textwidth}
                \centering
\begin{tiny}
$$
\raisebox{.25in}{\Bigg(}
\hspace{.05in}
\begin{tableau}
:.{2}.{3}.{6}\\
:.{4}.{7}\\
:.{8}\\
:;\\
\end{tableau}
\hspace{.05in}
\raisebox{.1in}{,}
\hspace{.05in}
\begin{tableau}
:.{1}.{5}\\
:.{9}\\
:;\\
:;\\
\end{tableau}
\hspace{.05in}
\raisebox{.25in}{\Bigg)}
$$
\end{tiny}
\vspace{-.2in}
\caption*{    \hspace{.2in} (b)}
\end{subfigure}
\hspace{.4in}
\begin{subfigure}{0.2\textwidth}
                \centering
                \begin{tiny}
$$
\begin{tableau}
:.{0}.{0}>{1}\\
:.{0}>{2}\\
:^{3}^{4}\\
:;\\
\end{tableau}
$$
\end{tiny}
\vspace{-.2in}
\caption*{(c)}
\end{subfigure}

\caption{(a) The Young diagram of the partition $(4,3^2,1)$, (b) a standard bitableau of shape $((3,2,1),(2,1))$, and (c) a standard domino tableau of rank $2$.}
\end{figure}

If the rightmost square $s$ of a row in a Young diagram $[\lambda]$  can be removed while leaving another Young diagram $[\lambda]_s$, it will be called a {\it corner of} $[\lambda]$.  Beginning with $[\lambda]$, one can start successively removing corners until this process inevitably terminates after $n$ steps with the empty partition.  If the square removed at the $i$th step of this procedure is labeled with the number $n-i+1$, the result is  called a {\it standard Young tableau}.  For a tableau $T$, we will write $sh(T)$ for the underlying partition, $|T|$ for $|sh(T)|$, $SYT(\lambda)$ for the set of all standard Young tableaux of shape $\lambda$, and $SYT(n)$ for the set of standard Young tableaux with $n$ boxes.

We will call an ordered pair of partitions $(\lambda, \mu)$ a {\it bipartition} of $n$ if $|\lambda|+|\mu|=n$. A square of $([\lambda],[\mu])$ is a corner if it is a corner of either $[\lambda]$ or $[\mu]$.  Successive removal of corners starting with $([\lambda],[\mu])$  terminates after $n$ steps, and if the square removed at the $i$th step is labeled with $n-i+1$, the result will be called a {\it standard bitableau}.  The shape of a bitableau is the pair of its underlying partitions; we will write $SBT(\lambda, \mu)$ for the set of all standard bitableaux of shape $(\lambda, \mu)$, and $SBT(n)$ for the set of standard bitableaux with $n$ boxes.

Two squares of a Young diagram are {\it adjacent} if they share a common side.  Adjacent squares $s, t$ in $[\lambda]$ form a {\it domino corner} if $s$ is a corner for $[\lambda]$ and $t$ is a corner for $[\lambda]_s$.
Beginning with $[\lambda]$, one can start successively removing domino corners and continue until this is no longer possible, say after $n$ steps.  The resulting shape is a staircase partition $\delta_r= (r, r-1, r-2, \ldots 1)$ for some $r \geq 0$ and is independent of the order of removal of domino corners, see \cite{macdonald:symmetric}.  The partition $\delta_r$ is known as the $2$-{\it core} of $\lambda$.  If the squares of $[\lambda]$ corresponding to the $2$-core are labeled with $0$ and the domino removed at the $i$th step is labeled with $n-i+1$, the result is a {\it standard domino tableau} of rank $r$.  The set of all standard domino tableaux of shape $\lambda$ with $2$-core $\delta_r$ will be denoted by $SDT_r(\lambda)$ while $SDT_r(n)$ will denote the set of all standard domino tableaux consisting of the $2$-core $\delta_r$ and $n$ dominos. We will call the set of squares in a domino tableau $T$ labeled with 0 the {\it core} of $T$.

\subsection{Robinson-Schensted maps}

Consider a permutation $w \in S_n$.  We will write it in one-line notation as $w_1 w_2 \ldots w_n$ with each entry $w_i \in \mathbb{N}_n$.  The classical Robinson-Schensted map establishes a bijection between permutations in $S_n$ and same-shape pairs of standard Young tableaux in $SYT(n) \times SYT(n)$ via an insertion and a recording algorithm.  We assume the reader is familiar with the basics; details can be found in \cite{fulton:tableaux} or \cite{stanley:enumerative}.  We will write $RS(w) = (P(w), Q(w))$ for the image of a permutation under this map.

A {\it signed permutation}  is a permutation together with a choice of sign for each of its entries.  We will again use one-line notation, using a bar over a letter to denote the choice of a negative sign.  The set of signed permutations on $n$ letters forms a group under composition and multiplication of signs;  it is isomorphic to the hyperoctahedral group $H_n = \mathbb{Z}_2 \wr S_n$ and is generated by
    $$s_i  = 1 \; 2 \, \ldots  i+1 \hspace{.1in} i \, \ldots  n \text{, \hspace{.2in} and \hspace{.2in}}
    t  = \overline{1} \;  2\, \ldots  n $$
for $1 \leq i < n$.  Let $\ell$ be the length function on $H_n$ defined in terms of this generating set and write $sign(w) = (-1)^{\ell(w)}.$

We are interested in two generalizations of the Robinson-Schensted map in this setting.  The first establishes a map
$$G_\infty:  H_n \longrightarrow SBT(n) \times SBT(n)$$
that is a bijection onto same-shape pairs of bitableaux.  Given a signed permutation, the insertion bitableau for $w \in H_n$ is constructed by a variant of the classical insertion algorithm.  Positive letters are inserted into the first tableau and negative into the second following their order of appearance in the one line notation for $w$.  The recording bitableau tracks the shape of the insertion bitableau at each step.  See \cite{stanley:some}.

\begin{example}
Consider the signed permutation $w= (\overline{4} \, \overline{3} 2 1) \in H_4.$  The sequence of bitableaux constructed by successive insertion of the letters of $w$ is:

\begin{tiny}
$$
\raisebox{.10in}{$( \varnothing , \varnothing) \rightarrow (\varnothing, $}
\hspace{.05in}
\raisebox{.10in}{\framebox[1.8\width]{4}}
\hspace{.02in}
\raisebox{.10in}{$)
\rightarrow$}
\raisebox{.1in}{$
\Big{(}\varnothing ,$}
\hspace{.05in}
\begin{tableau}
:.{3}\\
:.{4}\\
\end{tableau}
\hspace{.05in}
\raisebox{.10in}{$\Big{)}
\rightarrow
\Big{(}$}
\hspace{.05in}
\begin{tableau}
:.{2}\\
:;\\
\end{tableau}
\hspace{.05in}
\raisebox{.10in}{,}
\hspace{.05in}
\begin{tableau}
:.{3}\\
:.{4}\\
\end{tableau}
\hspace{.05in}
\raisebox{.10in}{$\Big{)}
\rightarrow
\Big{(}$}
\hspace{.05in}
\begin{tableau}
:.{1}\\
:.{2}\\
\end{tableau}
\hspace{.05in}
\raisebox{.10in}{,}
\hspace{.05in}
\begin{tableau}
:.{3}\\
:.{4}\\
\end{tableau}
\hspace{.05in}
\raisebox{.10in}{\Big{)}}
$$
\end{tiny}
\hspace{-.05in}
Keeping track of the shapes appearing in this sequence, we can construct another bitableau of the same shape obtaining:
$$\raisebox{.08in}{$G_\infty (w) = $
\Big{(}
\Big{(}
\hspace{.01in}}
\begin{tiny}
\begin{tableau}
:.{1}\\
:.{2}\\
\end{tableau}
\hspace{.05in}
,
\hspace{.05in}
\begin{tableau}
:.{3}\\
:.{4}\\
\end{tableau}
\hspace{.05in}
\end{tiny}
\raisebox{.08in}{\Big{)}}
\hspace{.05in},
\hspace{.05in}
\raisebox{.08in}{\Big{(}}
\hspace{.05in}
\begin{tiny}
\begin{tableau}
:.{3}\\
:.{4}\\
\end{tableau}
\hspace{.05in}
,
\hspace{.05in}
\begin{tableau}
:.{1}\\
:.{2}\\
\end{tableau}
\end{tiny}
\hspace{.05in}
\raisebox{.08in}{\Big{)} \Big{)}}.
$$
\end{example}

The second generalization of the classical Robinson-Schensted algorithm to the hyperoctahedral groups map has image within same-shape pairs of domino tableaux.  In fact, for every non-negative integer $r$, there is a map
$$G_r:  H_n \longrightarrow SDT_r(n) \times SDT_r(n)$$
that is a bijection onto same-shape pair of domino tableaux of rank $r$. Starting with the diagram $[\delta_r]$, a tableau is constructed via a domino insertion procedure inspired by the classical algorithm.  Positive letters are inserted as horizontal dominos in the first row of the tableau while negative ones are inserted as vertical dominos in its first column.   As long as the two types of dominos do not interact, the procedure is very similar to classical insertion; when they do, a more complicated bumping procedure becomes necessary.

\begin{example}  Let $r=2$ and $w= (\overline{4} \, \overline{3} 2 1) \in H_4.$  The sequence of domino tableaux constructed by successive insertion of the letters of $w$ into the 2-core $[\delta_2]$ is:
$$
\begin{tiny}
\begin{tableau}
:.{0}.{0}\\
:.{0}\\
:;\\
:;\\
\end{tableau}
\hspace{.1in}
\raisebox{.2in}{$\longrightarrow$}
\hspace{.1in}
\begin{tableau}
:.{0}.{0}\\
:.{0}\\
:^4\\
:;\\
\end{tableau}
\hspace{.1in}
\raisebox{.2in}{$\longrightarrow$}
\hspace{.1in}
\begin{tableau}
:.{0}.{0}\\
:.{0}^4\\
:^3\\
:;\\
\end{tableau}
\hspace{.1in}
\raisebox{.2in}{$\longrightarrow$}
\hspace{.1in}
\begin{tableau}
:.{0}.{0}>{2}\\
:.{0}^4\\
:^3\\
:;\\
\end{tableau}
\hspace{.1in}
\raisebox{.2in}{$\longrightarrow$}
\hspace{.1in}
\begin{tableau}
:.{0}.{0}>{1}\\
:.{0}>2\\
:^3>4\\
:;\\
\end{tableau}
\end{tiny}
$$
Keeping track of the shapes appearing in this sequence, we can construct another domino tableau of the same shape obtaining:
$$\raisebox{.18in}{$G_2 (w) = $
\hspace{.1in}
\Bigg{(}}
\hspace{.05in}
\begin{tiny}
\begin{tableau}
:.{0}.{0}>{1}\\
:.{0}>2\\
:^3>4\\
:;\\
\end{tableau}
\hspace{.1in}
\raisebox{.05in}{,}
\hspace{.1in}
\begin{tableau}
:.{0}.{0}>{3}\\
:.{0}^{2}^{4}\\
:^{1}\\
:;\\
\end{tableau}
\end{tiny}
\raisebox{.18in}{
\Bigg{)}}
$$
\end{example}

An initial version of the hyperoctahedral Robinson-Schensted maps $G_0$ and $G_1$ is due to D.~Barbasch and D.~Vogan \cite{barbasch:vogan}, but was only described later in terms of domino insertion by D.~Garfinkle \cite{garfinkle1}.  Van Leeuwen showed that the bijection holds for all $r$ and described the map using growth diagrams \cite{vanleeuwen:RS}.  For a more formal description of the $G_r$ where all the details of the insertion and bumping procedures may be found, see \cite{garfinkle1} or \cite{sw}.

When $r$ is sufficiently large relative to $n$, inserted dominos corresponding to the positive and negative letters of $w$ do not interact and it is easy to see that it is possible to recover $G_\infty(w)$ from $G_r(w)$.  Thus in this sense, $G_\infty$ is an asymptotic version of the $G_r$.

\begin{figure}
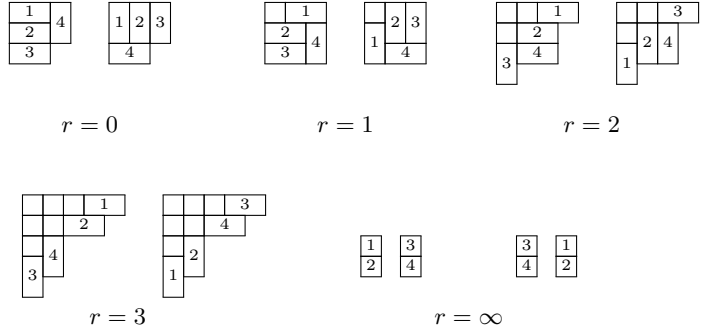


        \begin{subfigure}[b]{0.3\textwidth}
                \centering
                $$
                \begin{tiny}
\begin{tableau}
:>{1}^{4}\\
:>2\\
:>3\\
:;\\
\end{tableau}
\hspace{.1in}
\hspace{.1in}
\begin{tableau}
:^{1}^{2}^{3}\\
:;\\
:>{4}\\
:;\\
\end{tableau}
\end{tiny}
$$
                \caption*{$r=0$}
        \end{subfigure}
                \hspace{.0in}
        \begin{subfigure}[b]{0.2\textwidth}
                \centering
$$
 \begin{tiny}
\begin{tableau}
:.{}>{1}\\
:>2^{4}\\
:>3\\
:;\\
\end{tableau}
\hspace{.1in}
\hspace{.1in}
\begin{tableau}
:.{}^{2}^{3}\\
:^{1}\\
:;>{4}\\
:;\\
\end{tableau}
\end{tiny}
$$
                \caption*{$r=1$}
                \label{fig:gull}
        \end{subfigure}
\hspace{.2in}
        \begin{subfigure}[b]{0.2\textwidth}
                \centering
$$
 \begin{tiny}
\begin{tableau}
:.{}.{}>{1}\\
:.{}>2\\
:^3>4\\
:;\\
\end{tableau}
\hspace{.1in}
\hspace{.1in}
\begin{tableau}
:.{}.{}>{3}\\
:.{}^{2}^{4}\\
:^{1}\\
:;\\
\end{tableau}
\end{tiny}
$$
                \caption*{$r=2$}
                \label{fig:gull}
        \end{subfigure}

\vspace{.2in}

        \begin{subfigure}[b]{0.2\textwidth}
                \centering
$$
 \begin{tiny}
\begin{tableau}
:.{}.{}.{}>{1}\\
:.{}.{}>2\\
:.{}^4\\
:^3\\
\end{tableau}
\hspace{.1in}
\hspace{.1in}
\begin{tableau}
:.{}.{}.{}>{3}\\
:.{}.{}>4\\
:.{}^2\\
:^1\\
\end{tableau}
\end{tiny}
$$
                \caption*{$r=3$}
                \label{fig:gull}
        \end{subfigure}
\hspace{.5in}
                \begin{subfigure}[b]{0.3\textwidth}
                \centering
$$
 \begin{tiny}
\begin{tableau}
:.1\\
:.2\\
\end{tableau}
\hspace{.1in}
\begin{tableau}
:.3\\
:.4\\
\end{tableau}
\hspace{.5in}
\begin{tableau}
:.3\\
:.4\\
\end{tableau}
\hspace{.1in}
\begin{tableau}
:.1\\
:.2\\
\end{tableau}
\end{tiny}
$$
                \caption*{$r=\infty$}
                \label{fig:gull}
        \end{subfigure}

\caption{Images of $w= (\overline{4} \, \overline{3} 2 1)$ under the domino Robinson-Schensted maps $G_r$}
\end{figure}

\subsection{Tableaux statistics}\label{subsection:stats}
Our ultimate goal is to be able to read off the sign character of a signed permutation from its image under the various Robinson-Schensted maps.  To do so, we first need to define and extend a few tableaux statistics.

\begin{definition}
Let $T \in SYT(n)$.  A pair of entries $(i,j)$ is an {\it inversion} in $T$ if $j<i$ and $j$ is contained in a row strictly below the row of $i$.
\end{definition}

We first extend the definitions of inversions and sign to bitableaux as well as domino tableaux.

\begin{definition}
Let $T = (T_1,T_2) \in SBT(n)$.   A pair of entries $(i,j)$ is an {\it inversion} in $T$ if it is either an inversion in the standard Young tableaux $T_1$ or $T_2$, or $j<i$ and $j$ is contained in $T_1$ while $i$ is contained in $T_2$.  We define $spin(T) = |T_2|/2$.
\end{definition}

\begin{definition} Let $T\in SDT(n)$.  We will say that a square of $T$ is {\it marked} if the sum of its coordinates is even. A pair of positive entries $(i,j)$ is an {\it inversion} in $T$ if $j<i$ and $j$ labels a marked square in a row strictly below the marked square with label $i$.
\end{definition}

For a standard Young, bi-, or domino tableau $T$, we will write $Inv(T)$ for the set of its inversions and $inv(T)$ for the the cardinality of this set.  The {\it sign} of the tableau $T$ will then be defined as $sign(T)= (-1)^{inv(T)}$.
Note that  when applied to domino tableaux, the present notion of sign differs in general from the traditional one as defined in \cite{white}, for example.  In particular, using the present definition, the sign of a domino tableau depends on more than just the underlying domino tiling of the domino tableau shape, see \cite[Prop. 9]{white}.

\begin{example} Consider the following three tableaux:

$$
\raisebox{.25in}{$T=$}
\hspace{.1in}
\begin{tiny}
\begin{tableau}
:.{1}.{2}.{4}\\
:.{3}.{6}\\
:.{5}\\
:;\\
\end{tableau}
\end{tiny}
\hspace{.3in}
\raisebox{.25in}{$S=$
\hspace{0in} \Big(}
\hspace{.1in}
\begin{tiny}
\begin{tableau}
:.{2}.{3}.{6}\\
:.{4}.{7}\\
:.{8}\\
:;\\
\end{tableau}
\end{tiny}
\hspace{.05in}
\raisebox{.2in}{,}
\hspace{.05in}
\begin{tiny}
\begin{tableau}
:.{1}.{5}\\
:.{9}\\
:;\\
:;\\
\end{tableau}
\end{tiny}
\hspace{.05in}
\raisebox{.25in}{\Big)}
\hspace{.4in}
\raisebox{.25in}{$U=$}
\hspace{.1in}
\begin{tiny}
\begin{tableau}
:.{0}.{0}>{1}\\
:.{0}>{2}\\
:^{3}^{4}\\
:;\\
\end{tableau}
\end{tiny}
$$
According to the above definitions, their sets of inversions are $Inv(T)= \{ (4,3), (6,5) \}$ $Inv(S)= \{ (2,1), (3,1),(4,1),(7,1),(8,1),(6,4),(6,5),(7,5),(8,5)\}$, and $Inv(U)= \varnothing$.

\end{example}

We also define a few statistics special to domino tableaux.  Let $v(T)$ be the number of vertical dominos in $T$ and let $spin(T) = v(T)/2$.  For a signed permutation $w$, the {\it total color} $tc(w)$ is the number of negative letters in its one line notation.  This statistic is particularly well-behaved with respect to the domino Robinson-Schensted maps $G_r$.

\begin{theorem}[\cite{sw}, \cite{lam:growth}] Consider a signed permutation $w \in H_n$ and further let $G_r(w) = (P,Q)$ be its image under  the domino Robinson-Schensted map $G_r$ among same-shape pairs of domino tableaux with $2$-core $\delta_r$.  Then
$$tc(w) = spin(P) + spin(Q).$$
\label{theorem:colortospin}
\end{theorem}

\noindent
This result is clear when $r \geq n-1$.  It was verified for $r=0$ in \cite{sw} and extended to all $r$ in \cite{lam:growth}.  Of particular interest to us is the immediate observation that the sum of the spins of the recording and tracking tableaux for the maps $G_r$ is independent of $r$.

For a standard domino tableau $T$, let $eh(T)$ and $ev(T)$ denote the number of horizontal dominos in even index rows and vertical dominos is even index columns of $T$, respectively.  If we let $d(T)$ denote the number of squares with a positive label which lie both in an even row and even column of $T$, then $eh(T)+ev(T)= d(T)$.

\subsection{The type $A$ sign character}

We can now state the result of A.~Reifegerste and J.~Sj\"ostrand which our main results generalize.

\begin{theorem}[\cite{reifegerste},\cite{sjostrand}]
Consider a permutation $w \in S_n$ and let $RS(w) = (P,Q)$ be its image under the classical Robinson-Schensted map.  Then
$$ sign(w) = (-1)^{e} \cdot sign(P) \cdot sign(Q)$$
where $e=e(P)$ is the sum of the lengths of all the even-index rows of $P$.
\end{theorem}

The following is a special case of a result on the sign characters of the complex reflection groups $G(r,p,n)$.

\begin{theorem}[\cite{mps}]
Consider $w \in H_n$ and let $G_\infty(w) = (P,Q)$ be its image among same-shape standard bitableaux, where $P=(P_1,P_2)$ and $Q=(Q_1,Q_2)$.  Then
$$ sign(w) = (-1)^{e} \cdot (-1)^{spin(P)+spin(Q)}  \cdot sign(P) \cdot sign(Q)$$
where $e=e(P_1)+e(P_2)$ is the sum of the lengths of all the even-index rows of the constituent tableaux of $P$.
\label{theorem:mps}
\end{theorem}

\subsection{Cycles}
The most technical aspect of this work lies in the notion of a cycle in a domino tableau.
First defined in \cite{garfinkle1}, cycles have appeared in various settings, including \cite{carre-leclerc}, \cite{vanleeuwen:edge}, \cite{vanleeuwen:bijective}, and \cite{pietraho:relation}.  We provide a brief introduction, beginning with a few definitions.

For a standard domino tableau $T \in SDT_r(n)$, we will say the square $s_{ij}$ in row $i$ and column $j$ of $T$ is {\it variable} when $i+j \equiv r \, \mod 2$, otherwise,
we will call it {\it fixed}.  Following \cite{garfinkle1}, we further differentiate variable squares by saying
$s_{ij}$ is of {\it type $X$} if $i$ is odd and of {\it type $W$} otherwise.
Write $D(k, T )$ for the domino labeled by the positive integer $k$ in $ T$ and $supp \,  D(k, T )$ for its underlying squares. Write $label \, s_{ij}$ for the integer label of $s_{ij}$ in $T$  and let $label \, s_{ij}$ = 0 if either $i$ or $j$ is less than or equal to
zero, and $label \, s_{ij} = \infty$ if $i$ and $j$ are positive but $s_{ij}$ is not a square in $T$ .

\begin{definition}
Let $supp \, D(k,T)= \{s_{ij},s_{i+1,j}\}$ or
$\{s_{i,j-1},s_{ij}\}$ and suppose that the square  $s_{ij}$ is fixed. Define
a new domino $D'(k)$ labeled by the integer $k$ by letting $supp \,
D'(k,T)$ equal to
    \begin{enumerate}
        \item $\{s_{ij}, s_{i-1,j}\}$     if $k< label \, s_{i-1,j+1}$
        \item $\{s_{ij}, s_{i,j+1}\}$    if $k> label \, s_{i-1,j+1}$
    \end{enumerate}
If  $supp \, D(k,T)= \{s_{ij},s_{i-1,j}\}$
or $\{s_{i,j+1},s_{ij}\}$ and the square  $s_{ij}$ is fixed, then
define $supp \, D'(k,T)$ to be
    \begin{enumerate}
        \item $\{s_{ij},s_{i,j-1}\} $    if $k< label \, s_{i+1,j-1}$
        \item $\{s_{ij},s_{i+1,j}\}$         if $k> label \, s_{i+1,j-1}$
    \end{enumerate}
\end{definition}

\begin{definition}
For $T \in SDT_r(n)$, the {\it cycle $c=c(k,T)$ through $k$} is the set of integers defined by the condition that $l \in c$ if
either
    \begin{enumerate}
        \item $l=k$,
        \item $supp \, D(l,T) \cap supp \, D'(m,T) \neq
        \emptyset$ for some $m \in c$, or
        \item $supp \, D'(l,T) \cap supp \, D(m,T) \neq
        \emptyset$ for some $m \in c$.
    \end{enumerate}
We identify the labels contained in a cycle with their underlying dominos.
\end{definition}

If $c$ is a cycle in $T$, it is possible to construct a tableau $MT(T,c)$ by replacing every domino $D(l,T) \in c$ by the shifted domino $D'(l,T)$ defined above.  This map produces a standard domino tableau,  preserves the labels of the fixed squares of $T$, and changes the labels of the variable squares in $c$.
The shape of $MT(T,c)$ either equals the original shape of $T$, or one square will be removed (or added to the core) and one will be added.
In the first case, the cycle $c$ is called {\it closed}; otherwise, it is called {\it open}.
For an open cycle $c$ of a tableau $T$, we will write $S_b(c)$ for the square that has been removed (or added to the core) by moving through $c$.  It is the {\it beginning square} of the $c$. Similarly, we will write $S_f(c)$ for the square that is added to the shape of $T$, the {\it final square of $c$}. Note that $S_b(c)$ and $S_f(c)$ are always variable squares.

\begin{example}  Below are two diagrams of a standard domino tableau of rank $2$, the first unadorned, and then with its cycles highlighted and the final squares of the open cycles displayed as dashed boxes.  There is one closed cycle, $c=\{11, 12\}$.

\vspace{.1in}

\begin{center}
    \raisebox{.15in}{\includegraphics[width=.85in]{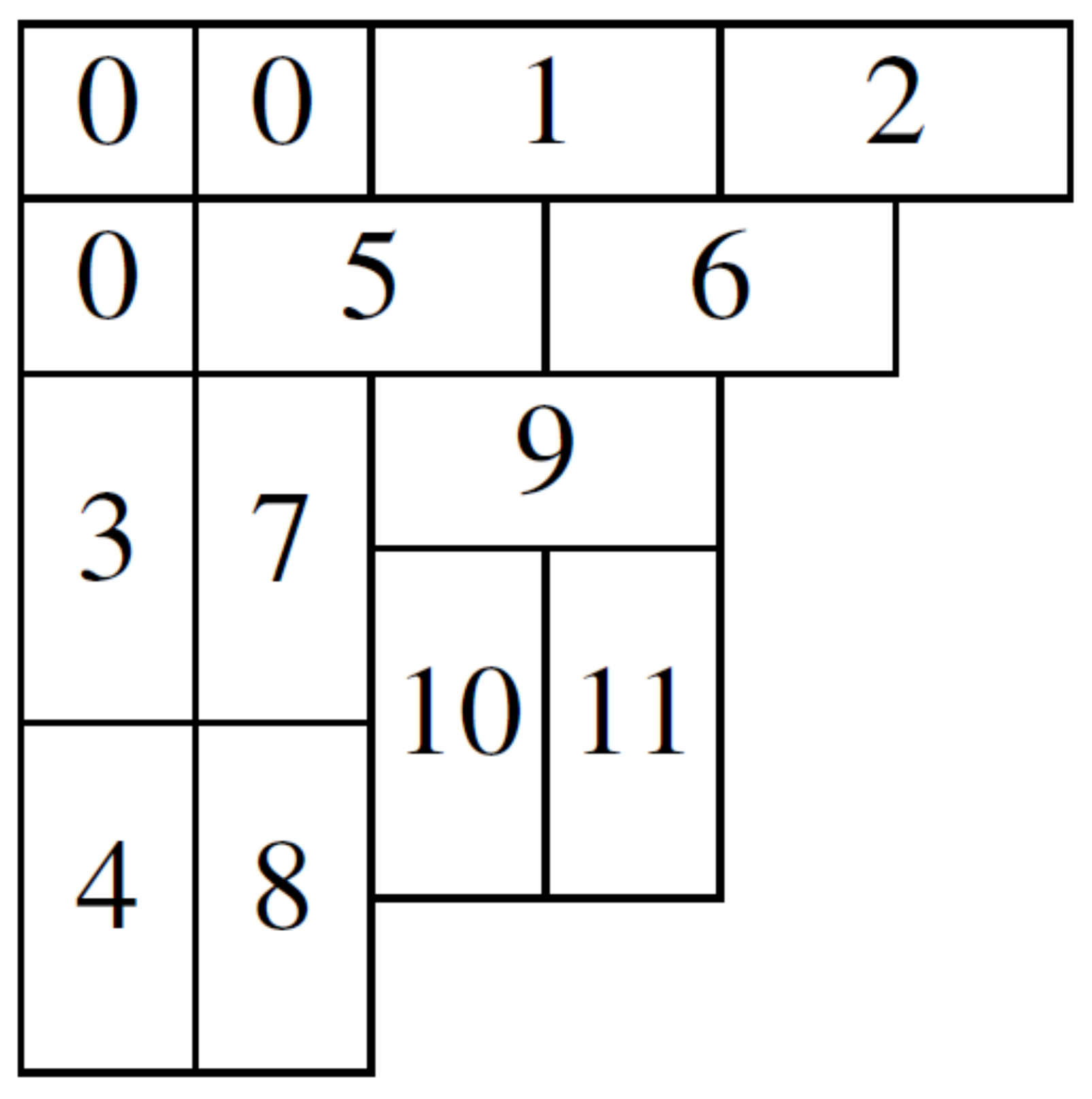}} \hspace{.2in}
    \includegraphics[width=1in]{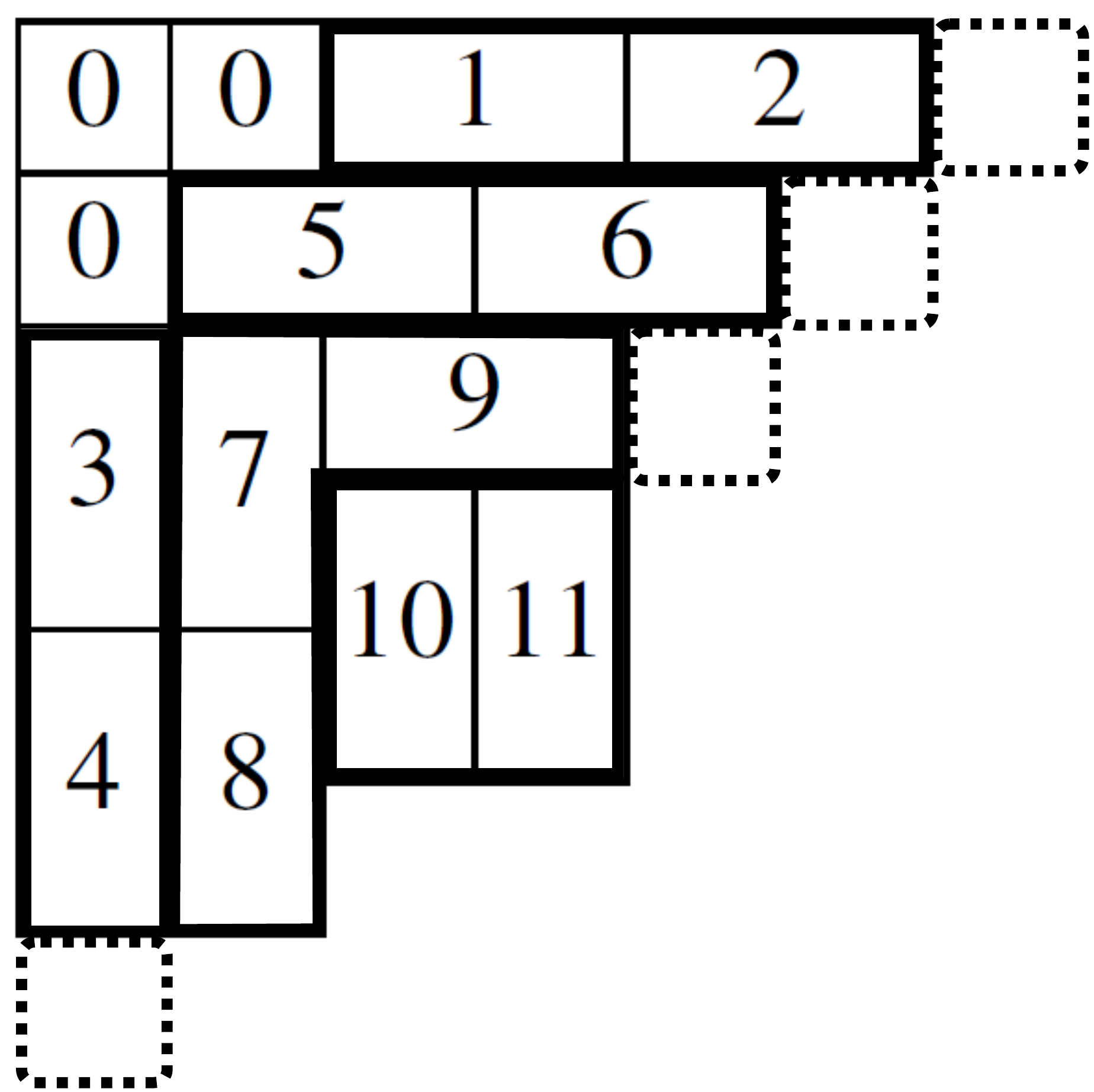}
\end{center}

\end{example}

Each square adjacent to the core in $T \in SDT_r(n)$ is the beginning square for some open cycle.  We call the set of all such open cycles $\Delta(T)$.  One fact about moving through open cycles will be especially relevant and is apparent in the example above.

\begin{proposition}[\cite{garfinkle1},\cite{pietraho:knuth}] Consider $T \in SDT_r(n)$.  If $c \in \Delta(T)$, then the variable squares $S_b(c)$ and $S_f(c)$ are both of type $X$ or both of type $W$.  In particular, $S_b(c)$ lies in an even row and column of $T$ iff $S_f(c)$ does as well.
\label{proposition:cycles}
\end{proposition}

The order in which one applies the moving through map to cycles in a set $U$ is immaterial by \cite{garfinkle1}(1.5.29), allowing us to write $MT(T,U)$ for the tableau obtained by moving through all of the cycles in the set $U$.

\section{Sign of colored permutations}

Based on the definitions of domino tableaux statistics enumerated in Section \ref{subsection:stats} we are ready to state a formula for reading the sign of a colored permutation from its image under any of the domino Robinson-Schensted maps.

\begin{theorem}
Consider a signed permutation $w \in H_n$ and let $G_r(w) = (P,Q)$ be its image among same-shape standard domino tableaux of rank $r$.  Then
$$ sign(w) = (-1)^{d} \cdot (-1)^{spin(P)+spin(Q)} \cdot sign(P) \cdot sign(Q)$$
where $d=d(P)$ denotes the number of non-core squares which lie concurrently in an even row and column of $P$, and $spin$ denotes the spin of a tableau.
\label{theorem:main}
\end{theorem}

Our first goal is to verify the theorem for involutions in $H_n$.  There are two main tools, the sign formula for colored permutations under $G_\infty$ derived from  \cite{mps} and a description of the relationship between the maps $G_r$ and $G_{r+1}$ obtained in \cite{pietraho:relation}.  When $r$ is large relative to $n$, the relationship between $G_\infty$ and $G_r$ is simple and it is a trivial task to translate one sign formula into the other.  Using the map of \cite{pietraho:relation}, we then extend the result on involutions to all $r$.

Under the maps $G_r$, the left and right tableaux for an involution in $H_n$ coincide, see \cite{vanL:RS}.  To complete our proof, we examine the behavior of the sign character under the plactic relations on $H_n$ obtained in \cite{taskin:plactic}.  As plactic relations generate the equivalence classes of having the same left tableau under $G_r$, this extends the theorem to all of $H_n$.

\subsection{Involutions}  The goal of this section is to verify the claimed sign formula for involutions in $H_n$.  We follow the approach outlined above and start by translating the formula for the asymptotic map $G_\infty$ to the maps $G_r$ for $r \geq n-1$.

\begin{lemma} Let $i \in H_n$ be an involution and write $G_r(i) = (P,P)$.  For $r \geq n-1$,
$$sign(i) = (-1)^d \cdot (-1)^{2 spin(P)}$$
where $d=d(P)$ denotes the number of non-core squares which lie concurrently in an even row and column of $P$, and $spin$ denotes the spin of a tableau.
\label{lemma:larger}
\end{lemma}
\begin{proof}
Let $G_\infty(i) = (R,R)$.
   By Theorem \ref{theorem:mps}, if we write $R=(R_1,R_2)$, then
$ sign(i) = (-1)^{e + |R_2|}$
where $e=e(R_1)+e(R_2)$ is the sum of the lengths of all the even-index rows of the constituent tableaux of $R$.  Now note that $tc(i)=|R_2|= 2 spin(P)$, so it remains to show that $d=e$.  The squares in even-indexed rows of $R_1$ correspond to horizontal dominos in even rows of $P$ and the squares in even-indexed rows of $R_2$ correspond to vertical dominos in even columns of $P$.  Equality follows.
\end{proof}

Next, we extend this formula to all values of $r$.   Let $r$ and $r'$ be non-negative integers and
suppose that $T \in SDT_r(n)$.  Following \cite{pietraho:relation}, we define a map $$t_{r,r'}:SDT_r(n) \rightarrow
SDT_{r'}(n)$$  by setting $t_{r,r'}(T) = T'$ whenever $G_r^{-1}(T,T) = G_{r'}^{-1}(T',T')$.  When $r$ and $r'$ are consecutive integers, this map has a particularly simple description in terms of cycles in a domino tableau. Given $T \in SDT_r(n)$, it is easy to produce a domino tableau of rank $r+1$ by moving through open cycles;  simply move thorough all the open cycles in $\Delta(T)$.  It is clear that the $2$-core of the resulting tableau is $\delta_{r+1}$.  What is perhaps surprising is that this map coincides with $t_{r,r+1}$:

\begin{theorem}[\cite{pietraho:relation}] Let  $t_{r,r+1}:SDT_r(n) \rightarrow
SDT_{r+1}(n)$  be defined as above.  Then
$$t_{r,r+1}(T) = MT(T,\Delta(T)).$$
\label{theorem:relation}
\end{theorem}

\begin{example}  Consider the involution $i=(5 \, 9 \, \overline{7} \, \overline{11} \, 1 \, 6 \, \overline{3} \, \overline{10} \, 2 \, \overline{8} \, \overline{4}) \in H_{11}$.  Its image under the domino Robinson-Schensted algorithm $G_2$ is a pair of tableaux $(P,P)$ with $P$ as below.
$$
\raisebox{.3in}{$P=$}
\hspace{.1in}
\begin{tiny}
\begin{tableau}
:.0.0>1>2\\
:.0>5>6\\
:^3^7>{9}\\
:;;^{10}^{11}\\
:^4^{8}\\
:;\\
\end{tableau}
\end{tiny}
\hspace{.4in}
\raisebox{.3in}{$P'=$}
\hspace{.1in}
\begin{tiny}
\begin{tableau}
:.0.0.0>1>2\\
:.0.0>5>6\\
:.0^7>{9}\\
:^3;^{10}^{11}\\
:;^{8}\\
:^4\\
\end{tableau}
\end{tiny}
$$

\vspace{.2in}

\noindent
There are three open cycles in $\Delta(P)$, mainly $\{1,2\}$, $\{3,4\}$, and $\{5,6\}$.  Moving through all three produces the tableau $P'$.  As claimed by Theorem \ref{theorem:relation}, this is also the image of $i$ under $G_3$; mainly $G_3(i) = (P',P')$.

\end{example}

In order to extend the involution sign formula to all values of $r$, it suffices to check that our tableau statistics are well-behaved with respect to the maps $t_{r,r+1}$ for all values of $r$.

\begin{lemma} Let $i \in H_n$ be an involution and write $G_r(i) = (P,P)$.  For $r \geq 0$,
$$sign(i) = (-1)^d \cdot (-1)^{2 spin(P)}$$
where $d=d(P)$ denotes the number of non-core squares which lie concurrently in an even row and column of $P$, and $spin$ denotes the spin of a tableau.
\end{lemma}
\begin{proof}
Let $P'=t_{r,r+1}(P)$.  We verify that $d(P)=d(P')$ and $spin(P)=spin(P')$, showing that the right hand side of the claimed equation is independent of $r$.  Since the theorem holds for large $r$ by Lemma \ref{lemma:larger}, the result will follow.

First, note that $spin(P)=spin(P')$ since both equal $tc(i)$ by Theorem \ref{theorem:colortospin}.  Since $P'=MT(P,\Delta(P))$ by Theorem \ref{theorem:relation}, the difference between the shapes of the two tableaux are the beginning and final squares for the open cycles in $\Delta(P)$.  In this process, $d(P)$ is reduced by one for each cycle whose beginning square lies in an even row and an even column.  It increases by one for each cycle whose final square lies in an even row and an even column.  But by Proposition \ref{proposition:cycles}, cycles in $\Delta(P)$ whose final square lies in an even row and column are precisely those whose beginning square lies in an even row and an even column.  Thus $d(P)=d(P')$.
\end{proof}

\subsection{Extension to $H_n$}

In this section we complete the proof of Theorem \ref{theorem:main}.  Each of the Robinson-Schensted algorithms $G_r$ suggests a equivalence relation on $H_n$, with two colored permutations equivalent if and only if they share same left tableau in the image of $G_r$.  While this family of relations has significance in representation theory and the Kazhdan-Lusztig theory of cells, see \cite{garfinkle1} and \cite{bgil}, we have an opportunity to use it toward our more modest purpose.
In \cite{taskin:plactic}, M.~Ta\c{s}k{\i}n described a set of generators for each of the above equivalence relations.  To prove Theorem \ref{theorem:main}, we track the action of each generator on left tableaux as well as the sign of the corresponding colored permutation.

We reproduce the definitions of five operators on $H_n$ originally appearing in \cite{taskin:plactic}.  The first is derived from the original Knuth relations of \cite{knuth:pacific}. Precursors to the next two appear in \cite{bonnafe:iancu} and \cite{garfinkle2} as plactic relations for $G_\infty$ and $G_0$.  The final two are designed to deal with two specific situations appearing among domino tableaux, especially of higher rank.  Write $w = w_1 w_2 \ldots w_n$ for a colored permutation and adopt the convention that $\overline{\overline{z}}=z$.

\vspace{.1in}
\begin{quote}

\noindent
1. If $w_i < w_{i+2} <w_{i+1}$ or $w_i <w_{i-1}<w_{i+1}$ for some $i < n$, then

\vspace{.1in}

\begin{center}
\framebox{$D_1^r(w)=  w_1\ldots w_{i-1}(w_{i+1} w_{i}) w_{i+2} \ldots w_n.$}
\end{center}

\vspace{.1in}

\noindent
2. If $r>0$ and if there exists $0<i\leq r$ such that $w_i$ and $w_{i+1}$ have opposite signs, then

\vspace{.1in}

\begin{center}
\framebox{$D_2^r(w)= w_1\ldots w_{i-1}(w_{i+1} w_{i}) w_{i+2} \ldots w_n.$}
\end{center}

\vspace{.1in}

\noindent
3. Suppose that $|w_1| > |w_i|$ for all $1 < i \leq r + 2$ and $w_2 \ldots w_{r+2}$ is obtained by concatenating some positive decreasing sequence to the end of some negative increasing
sequence (or vice versa), where at least one of the sequences is nonempty. Then

\vspace{.1in}

\begin{center}
\framebox{$D_3^r(w)= \overline{w_1} w_2 \ldots w_n.$}
\end{center}

\vspace{.1in}

\noindent
4. Let $k \geq 1$ such that $t = (k + 1)(r + k + 1) \leq  m$ and suppose
$$w = \mathfrak{C}_1 \ldots \mathfrak{C}_k \mathfrak{C}_{k+1} \, z \, w_{t+1} \ldots w_n$$
where each $\mathfrak{C}_i$ is a sequence of the form
$\mathfrak{C}_i = a_{i,i+r}\ldots a_{i,1}b_{i,i}\ldots b_{i,1}$ for $1 \leq i \leq k$ and $\mathfrak{C}_{k+1} = a_{k+1,k+r} \ldots a_{k+1,1}.$
Further suppose that the integers $a_{i,j}$ and $b_{i,j}$, whenever they appear among $w_1\ldots w_{t-1} = \mathfrak{C}_1 \ldots  \mathfrak{C}_{k+1}$ satisfy the following conditions:
\begin{enumerate}[\phantom{WWW}]
    \item $a_{i,j} > 0$ and $b_{i,j} < 0$ (or vice versa)
    \item $|a_{i,j−1} | < |a_{i,j} | < |a_{i+1,j} |$ and $|b_{i,j−1} | < |b_{i,j} | < |b_{i+1,j} |$
    \item $|b_{i,i} | < |a_{i+1,r+i+1}| < |b_{i+1,i+1}|$ for all $i = 1,\ldots,k-1.$
\end{enumerate}
Let $n = \max\{|w_1|, \ldots, |w_{t-1}|\}$ and suppose that $w_t = z$ satisfies one of the following:
\begin{enumerate}
    \item $|b_{k,k} | = n$ and $z$ is an integer between $a_{k+1,1}$ and $b_{k,1}$
    \item $|a_{k+1,r+k} | = n$ and $z$ is an integer between $a_{k,1}$ and $b_{k,1}$
    \item $|a_{k+1,r+k} | = n$, $z$ is an integer between $a_{k,1}$ and $a_{k+1,1}$ and $|a_{k+1,i} | < |a_{k,i+1}|$ for some $1 < i \leq k-1.$
\end{enumerate}

\noindent
Then set  $\overline{\mathfrak{C}_k}  = \overline{b_{k,k}} a_{k,k+r} \ldots  a_{k,1} b_{k,k-1} \ldots b_{k,1}$ and define

\vspace{.1in}

\begin{center}
\framebox{$D_4^r= \mathfrak{C}_1 \ldots \mathfrak{C}_{k-1} \overline{\mathfrak{C}_k} \mathfrak{C}_{k+1}\, z \, w_{t+1}\ldots w_n$.}
\end{center}

\vspace{.1in}

\noindent
5. Let $k \geq 1$ such that $t = (k + 1)(r + k + 2) \leq  m$ and suppose
$$w = \mathfrak{C}_1 \ldots \mathfrak{C}_k \mathfrak{C}_{k+1} \, z \, w_{t+1} \ldots w_n$$
where each $\mathfrak{C}_i$ is a sequence of the form
$\mathfrak{C}_i = a_{i,i+r}\ldots a_{i,1}b_{i,i}\ldots b_{i,1}$ for $1 \leq i \leq k$ and $\mathfrak{C}_{k+1} = a_{k+1,k+r+1} \ldots a_{k+1,1} b_{k+1,k} \ldots b_{k+1,1}.$ Further suppose that the integers $a_{i,j}$ and $b_{i,j}$, whenever they appear among  $w_1\ldots w_{t-1} = \mathfrak{C}_1 \ldots  \mathfrak{C}_{k+1}$, satisfy the following conditions:
\begin{enumerate}[\phantom{WWW}]
    \item $a_{i,j} > 0$ and $b_{i,j} < 0$ (or vice versa)
    \item $|a_{i,j-1} | < |a_{i,j} | < |a_{i+1,j} |$ and $|b_{i,j-1} | < |b_{i,j} | < |b_{i+1,j} |$
    \item $|a_{i,r+i} | < |b_{i,i}| < |a_{i+1,r+i+1}|$ for all $i = 1,\ldots,k.$
\end{enumerate}
Let $n = \max\{|w_1|, \ldots, |w_{t-1}|\}$ and suppose that $w_t = z$ satisfies one of the following:
\begin{enumerate}
    \item $|a_{k+1,r+k+1} | = n$ and $z$ is an integer between $a_{k+1,1}$ and $b_{k+1,1}$
    \item $|b_{k+1,k} | = n$ and $z$ is an integer between $a_{k+1,1}$ and $b_{k,1}$
    \item $|b_{k+1,k} | = n$, $z$ is an integer between $b_{k,1}$ and $b_{k+1,1}$ and $|b_{k+1,i} | < |b_{k,i+1}|$ for some $1 < i \leq k-1.$
\end{enumerate}
Then we set $\overline{\mathfrak{C}_k} \overline{\mathfrak{C}_{k+1}} = a_{k,k+r} \ldots a_{k,1} \overline{a_{k+1,k+1+r}} b_{k,k} \ldots b_{k,1} a_{k+1,k+r} \\ \ldots a_{k+1,1} b_{k+1,k} \ldots b_{k+1,1}$ and define

\vspace{.1in}

\begin{center}
\framebox{$D_5^r = \mathfrak{C}_1 \ldots \mathfrak{C}_{k-1} \overline{\mathfrak{C}_k} \overline{\mathfrak{C}_{k+1}} z w_{t+1}\ldots w_n.$}
\end{center}
\end{quote}

\noindent
As promised, these generate the equivalence relations described above.

\begin{theorem}[\cite{taskin:plactic}]
Two colored permutations $w$ and $v$ have the same left tableau in the image of the map $G_r$ if and only if one can be obtained from the other via a sequence of operators $D^r_i$ for $i=1,2, \ldots, 5$
\end{theorem}

We proceed with a case by case examination of the action of each of the above generators on right tableaux and their effect on $sign$.  Consider a signed permutation $w \in H_n$ and let $G_r(w) = (P,Q)$ be its image among same-shape standard domino tableaux of rank $r$. Define
$$F(w) = (-1)^{d} \cdot (-1)^{spin(P)+spin(Q)} \cdot sign(P) \cdot sign(Q).$$

\subsubsection{} We first examine the operators $D_1^r$.  Note that $D_1^r(w) = s_i w$ for some $i$.  Thus in all cases $sign(D_1^r(w)) = - sign(w)$.  We verify that $F(D_1^r(w)) = - F(w)$.  Let $Q'$ be the right tableau of $D_1^r(w)$. There are two possibilities:
    \begin{quote}
    1. Suppose that the action of $D_1^r$ on the right tableau of $w$ exchanges a block of dominos
    $$
        \begin{tiny}
        \begin{tableau}
        :>k\\
        :^{l}^{m}\\
        :;\\
        \end{tableau}
        \end{tiny}
        \hspace{.1in}
        \raisebox{.1in}{\text{ with }}
        \hspace{.1in}
        \begin{tiny}
        \begin{tableau}
        :^k^{l}\\
        :;\\
        :>{m}\\
        \end{tableau}
        \end{tiny}
        \raisebox{.1in}{\text{ ,}}
        \hspace{.3in}
        \raisebox{.1in}{\text{ or }}
        \hspace{.3in}
        \begin{tiny}
        \begin{tableau}
        :^k>l\\
        :;>m\\
        :;\\
        \end{tableau}
        \end{tiny}
        \hspace{.1in}
        \raisebox{.1in}{\text{ with }}
        \hspace{.1in}
        \begin{tiny}
        \begin{tableau}
        :>{k}^m\\
        :>{l}\\
        :;\\
        \end{tableau}
        \end{tiny}
            \raisebox{.1in}{\text{ ,}}
    $$
    while keeping the rest of the tableau fixed.  For the sake of typesetting we are writing $l=k+1$ and $m=k+2$.  Then this operation preserves the $d$ statistic as well as $spin$.  We examine $sign(Q) = (-1)^{inv(Q)}$.  Recall that within domino tableaux, inversions are defined in terms of marked squares.  First assume that in fact $D_1^r$ changes
    $$
    \raisebox{.15in}{$S=$ }
        \begin{tiny}
        \begin{tableau}
        :^k>l\\
        :;>m\\
        :;\\
        \end{tableau}
        \end{tiny}
        \hspace{.1in}
        \raisebox{.15in}{\text{ into }}
        \hspace{.1in}
        \raisebox{.15in}{$S'=$ }
        \begin{tiny}
        \begin{tableau}
        :>{k}^m\\
        :>{l}\\
        :;\\
        \end{tableau}
        \end{tiny}
    $$
    and that the top left-most box of $k$ is marked.   Then $(k+2,k+1) \in Inv(Q')$ but not in $Inv(Q)$. The only other changes in $Inv(Q)$ occur when $p$ is a marked square lying in the same rows as $S$ in $Q$  but outside of $S$.  Then an inversion of the form $(x,p)$ or $(p,x)$ is exchanged for an inversion of the form $(y,p)$ or $(p,y)$ where $x,y \in \{k,k+1,k+2\}$.  Consequently,  $sign(Q) = - sign(Q')$.
    The other possibilities are similarly routine.

    \vspace{.1in}
    \noindent
    2. If the action of $D_1^r$ is not by exchange of one of the above configurations, then by \cite[2.1.19]{garfinkle2} and \cite[Prop 4.6]{pietraho:knuth}, $Q$ and $Q'$ differ by an exchange of labels of two consecutive dominos.  It is clear that this changes $inv(Q)$ by one.  The other statistics are constant.
    \end{quote}
In either case, $F(D_1^r(w)) = - F(w)$, as desired.

\subsubsection{} The case of the operator $D_2^r$ is very similar. Again, we have $D_2^r(w) = s_i w$ for some $i$ and consequently $sign(D_2^r(w)) = - sign(w)$.  The description of the action of this operator on $Q$ is implicit in the proof of \cite[Theorem 3.1]{taskin:plactic}.  It exchanges two consecutive dominos. As above, this changes $inv(Q)$ by one, holds the other statistics constant, and again, $F(D_2^r(w)) = - F(w)$, as desired.

\subsubsection{} In the case of $D_3^r$, we have $D_3^r(w) = t w$. Consequently $sign(D_3^r(w)) = - sign(w)$.  The action of $D_3^r$ on right tableaux is more intricate.  Our description is based on \cite[Cor 4.4 et seq.]{pietraho:knuth}.  Within $Q$, the operator exchanges the subtableaux

$$
\raisebox{.35in}{$S=$ \hspace{.2in}}
\begin{tiny}
\begin{tableau}
:.{}.{}.{}.{}.{}.{} \cdots >{1}\\
:.{}.{}.{}.{}.{}.{}\\
:.{}.{}.{}.{}.{}>{\alpha_3}\\
:.{}.{}.{}.{}>{\alpha_2}\\
:.{}.{}.{}^{\beta_1}>{\alpha_1}\\
:.{}.{}^{\beta_2}\\
:\vdots\\
:;\\
:^{\beta_p}\\
\end{tableau}
\end{tiny}
\hspace{.2in}
\raisebox{.35in}{
\text{ and }}
\hspace{.3in}
\raisebox{.35in}{$S'=$ \hspace{.2in}}
\begin{tiny}
\begin{tableau}
:.{}.{}.{}.{}.{}.{} \cdots >{\alpha_q}\\
:.{}.{}.{}.{}.{}.{} \\
:.{}.{}.{}.{}.{}>{\alpha_2}\\
:.{}.{}.{}.{}^{\,\beta_1}^{\alpha_1}\\
:.{}.{}.{}^{\,\beta_2}\\
:.{}.{}^{\beta_3}\\
:\vdots\\
:;\\
:^{1}\\
\end{tableau}
\end{tiny}
$$

\vspace{.2in}

\noindent
where the labels in $S$ and $S'$ coincide with $\mathbb{N}_{r+2},$ $\alpha_1= r+2$, and $p$ may equal zero.  The rest of $Q$ is fixed.  Note that the $d$ statistic is fixed by this operation and $spin(Q)$ changes by one.  We examine $sign(Q)$.  First assume that $r$ is odd and consequently the squares adjacent to the core are not marked squares.  We have to consider the effect this transformation has on the set of inversions in $Q$.  We consider two subsets, the inversions $Inv(S)$ among entries in $S$,  and those occurring between entries in $S$ and the rest of $Q$.  The order of the latter set is fixed by $D_3^r$ as labels in $S$ and $S'$ are just $\mathbb{N}_{r+2}$.  We compare the size of the former set in $S$ and $S'$.  By inspection,
$$Inv(S') \setminus Inv(S) =  \{(\alpha_i, 1)\}_{\alpha_i \neq 1} \cup \{(\beta_i, 1)\}_i.$$

\noindent
Thus $inv(S')= inv(S)+(r+1)$.  Since $r$ is odd, this means $sign(Q)=sign(Q')$.
Similar analysis applies in the case $r$ is even.  In either case, we have the equality $F(D_3^r(w)) = - F(w)$, as desired.

\subsubsection{}  At first glance, the operators $D_4^r$ and $D_5^r$ seem much more daunting than the prior three, but at least on the level of tableaux, they are in some sense just more intricate versions of $D_3^r$.
We first note that by the construction of $D_4^r$, adopting notation from its definition,
$$sign(D_4^r(w)) = (-1)^{k+r+1} sign(w).$$
The effect of the operators $D_4^r$ and $D_5^r$ on right tableau is described in the proof of \cite[Theorem 3.1]{taskin:plactic}.
There are four cases in our analysis of $F(D_4^r(w))$ distinguished by the sign of the $a_{ij}$ and the parity of $r$, which influences the choice of marked squares.  First assuming that the $a_{ij}$ are negative, the right tableaux $Q$ of $w$ and $Q'$ of $D_4^r(w)$ differ within the subtableaux

$$\raisebox{.5in}{$S=$ \hspace{.2in}}
\begin{tiny}
\begin{tableau}
:;;;;;;;>{l_1}\\
:;;;;;;>{l_2}\\
:\\
:;;;;;\cdot\\
:;;;;\cdot\\
:;;;\cdot\\
:;\\
:;>{l_{k-1}}\\
:>{l_k}\\
:>{t}\\
\end{tableau}
\end{tiny}
\hspace{.4in}
\raisebox{.5in}{and \hspace{.2in} $S'=$ \hspace{.2in}}
\hspace{.4in}
\begin{tiny}
\begin{tableau}
:;;;;;;;>{l_2}\\
:;;;;;;>{l_3}\\
:\\
:;;;;;\cdot\\
:;;;;\cdot\\
:;;;\cdot\\
:;\\
:;>{l_k}\\
:^{l_1}^{t}\\
:;\\
\end{tableau}
\end{tiny}
$$

\vspace{.1in}
\noindent
corresponding to the dominos inserted from the subword $\mathfrak{C}_k \mathfrak{C}_{k+1} \, z$.
Here we define $l=k+r+\sum_{i=1}^{k-1} 2i+r$, $t=(k+1)(r+k+1)$, and adopt the shorthand $l_j = l +j$.  Note that these tableaux are independent of the cases (1)-(3) in the definition of $D_4^r$.  The same subtableau results in all three.

It is clear that $d(Q) = d(Q')$ as this statistic only depends on the underlying tableau shape.  Further, $spin(Q)$ and $spin(Q')$ differ by one.  We analyze $inv(Q)$.  If $A$ is a subtableau of $B$, we will write $B \setminus A$ for the boxes in $B$ not included in $A$, and $B(t)$ for the subtableau of $B$ consisting of dominos with labels less than or equal to $t$.  Note that
\begin{align*}
inv(Q) & = inv(S) + inv(Q(t)\setminus S) + inv(Q(t) \setminus S , S)\\
 & + inv(Q\setminus Q(t)) + inv(Q\setminus Q(t), Q(t)).
\end{align*}
The only values in this decomposition that can potentially change in the transformation from $Q$ to $Q'$ are $inv(S)$ and $inv(Q(t) \setminus S , S)$.

 When $r$ is odd, then the top rightmost squares of $S$ and $S'$ are unmarked and
$inv(S')-inv(S) = k-1$, while $inv(Q'(t) \setminus S' , S') - inv(Q(t) \setminus S , S) = 0$.
When $r$ is even, then  then the top rightmost squares of $S$ and $S'$ are marked and
 $inv(S')-inv(S) = k$, while again
 $inv(Q'(t) \setminus S' , S') - inv(Q(t) \setminus S , S) = 0$.

When the $a_{ij}$ are positive, the tableaux $Q$ and $Q'$ differ in subtableaux that are transposes of $S$ and $S'$.  While the analysis is a little different, the above results are exactly the same: $inv(Q')-inv(Q)= k-1$ when $r$ is odd and $k$ when $r$ is even.  Hence in all the cases, $inv(Q') - inv(Q) \equiv k+r \mod{2}$ and $sign(Q') = (-1)^{k+r} sign(Q).$ Consequently,
    $$F(D_4^r(w)) = (-1)^{k+r+1}F(w).$$
and we find that $F(w)$ transforms in step with $sign(w)$, as desired.

\subsubsection{$D_5^r$}  Analysis of this operator follows a similar outline as that of $D_4^r$.  We first note that by its construction, again adopting notation from the definition of $D_5^r$,
$$sign(D_5^r(w)) = (-1)^{k+1} sign(w).$$
Again there are four cases in our analysis of $F(D_5^r(w))$.  First assuming that the $a_{ij}$ are positive, the right tableaux $Q$ of $w$ and $Q'$ of $D_5^r(w)$ differ within the subtableaux
$$\raisebox{.5in}{$S=$ \hspace{.2in}}
\begin{tiny}
\begin{tableau}
:;;;;;;;>{l_1}\\
:;;;;;;>{l_2}\\
:\\
:;;;;;\cdot\\
:;;;;\cdot\\
:;;;\cdot\\
:;\\
:;>{l'_0}\\
:>{l'_1}\\
:>{t}\\
\end{tableau}
\end{tiny}
\hspace{.4in}
\raisebox{.5in}{and \hspace{.2in} $S'=$ \hspace{.2in}}
\hspace{.4in}
\begin{tiny}
\begin{tableau}
:;;;;;;;>{l_2}\\
:;;;;;;>{l_3}\\
:\\
:;;;;;\cdot\\
:;;;;\cdot\\
:;;;\cdot\\
:;\\
:;>{l'_1}\\
:^{l_1}^{t}\\
:;\\
\end{tableau}
\end{tiny}
$$

\vspace{.1in}
\noindent
corresponding to the dominos inserted from the subword $\mathfrak{C}_k \mathfrak{C}_{k+1} \, z$.
Here we define $l=\sum_{i=1}^k (2i+r)$, $t = (k+1)(r+k+2)$, and for integers $j$, $l_j = l +j$, and $l'_j=l+k+r+j$.  Again, $d(Q)=d(Q')$ and $spin(Q)$ differs from $spin(Q')$ by $1$.  We analyse $inv(Q)$.

 When $r$ is odd, then the top rightmost squares of $S$ and $S'$ are marked and
 $inv(S')-inv(S) = k+r+1$, and
$inv(Q'(t) \setminus S' , S') - inv(Q(t) \setminus S , S) = 0$.
When $r$ is even, then  then the top rightmost squares of $S$ and $S'$ are unmarked and
 $inv(S')-inv(S) = k+r$, and again
$inv(Q'(t) \setminus S' , S') - inv(Q(t) \setminus S , S) = 0$.

When the $a_{ij}$ are negative, the tableaux $Q$ and $Q'$ differ in subtableaux that are transposes of $S$ and $S'$.  We again have $inv(Q')-inv(Q)= k+r+1$ when $r$ is odd and $k+r$ when $r$ is even with the other inversion statistics unchanged.  Hence in all the cases, $inv(Q') - inv(Q) \equiv k \mod{2}$ and $sign(Q') = (-1)^k sign(Q).$  Consequently,
$$F(D_5^r(w)) = (-1)^{k+1} F(w).$$
The proof of Theorem \ref{theorem:main} is complete.

\end{document}